\documentclass[reqno]{amsart}
\usepackage{amssymb,latexsym}
\usepackage{amsmath}
\usepackage{amsthm}
\usepackage{graphicx}
\usepackage{hyperref}
\usepackage{titletoc}
\numberwithin{equation}{section}
\newtheorem{theorem}{Theorem}[section]
\newtheorem{prop}[theorem]{Proposition}
\newtheorem{lemma}[theorem]{Lemma}

\theoremstyle{definition}

\newcommand{\Z}{\mathbb Z}

\newcommand{\N}{\mathbb N}

\newtheorem{remark}{Remark}

\begin{document}
\title[Prob. of visiting a point by critical Branching random walk in $\Z^4$]{An upper bound for the probability of visiting a distant point by critical branching random walk in $\Z^4$}
\author{Qingsan Zhu}
\address{~Department of Mathematics, University of British
Columbia,
Vancouver, BC V6T 1Z2, Canada}
\email{qszhu@math.ubc.ca}
\date{}
\maketitle

\begin{abstract}
In this paper, we study the probability of visiting a distant point $a\in \Z^4$ by critical branching random walk starting from the origin. We prove that this probability is bounded by $1/(|a|^2\log |a|)$ up to a constant.
\end{abstract}

\section{Introduction}

A branching random walk is a discrete-time particle system in $\Z^d$ as the following. Fix a distribution $\mu$ on $\N$, called offspring distribution, and $\theta$ on $\Z^d$, called jump distribution. At time $0$, there is a single particle at the origin $0\in\Z^d$. At each time step $n\in\N$, every particle, say at the site $x\in\Z^d$, gives birth to a random number of offspring (and dies afterwards), according $\mu$; each of these moves independently to a site according to distribution $x+\theta$. If the expectation of $\mu$ is one, we say that the branching random walk is critical.

The asymptotic behavior of the probability of visiting a distant point $a\in \Z^d$ by critical branching random walk in low dimensions ($d\leq 3$) was established recently by Le Gall and Lin (Theorem 7 in \cite{LL14}). Their theorem implies that (under some assumption about the branching random walk)
\begin{equation*}
P(\text{visiting } a)\asymp |a|^{-2} \quad \text{in } \Z^d \quad \text{for } d\leq 3;
\end{equation*}
where we write $f(a)\succeq g(a)$ ($f(a)\preceq g(a)$ resp.) if there exists a positive constant $c$ (only depending on $d$, the offspring distribution $\mu$ and the jump distribution $\theta$ of the critical branching random walk) such that $f(a)\geq cg(a)$ ($f(a)\leq cg(a)$ resp.) and $f(a)\asymp g(a)$ if $f(a)\succeq g(a)$ and $f(a)\preceq g(a)$.

It is also pointed out there (Section 5.1 in \cite{LL14}) that a simple calculation of the first moment and second moment gives
\begin{equation*}
P(\text{visiting } a)\asymp |a|^{2-d} \quad \text{in } \Z^d \quad \text{for } d\geq 5;
\end{equation*}
and
\begin{equation}\label{eq1}
P(\text{visiting } a)\succeq 1/(|a|^{2}\log |a|) \quad \text{in } \Z^4;
\end{equation}

It is believed:
\begin{equation}\label{eq2}
P(\text{visiting } a)\preceq 1/(|a|^{2}\log |a|) \quad \text{in } \Z^4.
\end{equation}

In this paper, we prove \eqref{eq2} under some assumption about $\theta$- we almost assume nothing about $\mu$, as long as $\mu$ is critical and nondegenerate i.e. $\mu(1)<1$. Let us state the main theorem.

\begin{theorem}\label{main}
Let $\mu$ be a critical and nondegenerate probability measure on $\N$ and $\theta$ be a probability measure on $\Z^4$ with zero mean and finite fifth moment (i.e. $E\theta=0$ and $E|\theta|^5<\infty$) , which is also not supported on a strict subgroup of $\Z^4$. If $\mathcal{S}$ is a critical branching random walk with offspring distribution $\mu$ and jump distribution $\theta$, then there exists a positive constant $C$ depending on $\mu$ and $\theta$, such that, for any $a\in\Z^4$ with $|a|$ sufficiently large, we have:
\begin{equation}\label{main_ineq}
P(\mathcal{S}\;\text{visiting }a)\leq C\cdot\frac{1}{|a|^2\log |a|}.
\end{equation}
\end{theorem}

\begin{remark}
Note that for \eqref{eq1} we need to assume that $\mu$ has finite variance. Hence if $\mu$ has finite variance in addition to the assumptions above, then:
\begin{equation}
P(\mathcal{S}\;\text{visiting }a)\asymp \frac{1}{|a|^2\log |a|}.
\end{equation}
\end{remark}

\begin{remark}
Update: based on the observation and result in this paper, the asymptotics of $P(\mathcal{S}\;\text{visiting }a)$ has been constructed in the forthcoming paper \cite{Z163}.
\end{remark}

\section{Proof of the main theorem}

Before the formal proof, let us talk a bit about the main idea. From simple calculation one can see that the expectation of the times of visiting $a$ is $G(a)=G(0,a)$ ($G$ is the Green function). Our assumptions about $\theta$ can guarantee $G(z)\asymp |z|^{-2}$ (see Theorem 4.3.5 in \cite{LL10}). If conditioned on visiting, the expectation of the visiting times is of order $\log |a|$, then we have \eqref{main_ineq}. In fact, we will prove that this is true with high probability.

Let us introduce some notation. Classically, branching random walk can be regarded as a random function $\mathcal{S}:V(T)\rightarrow \Z^d$, where $T$ is a random plane tree, i.e. rooted ordered tree, and $V(T)$ is the set of all vertices of $T$. In our case $T$ is a Galton-Watson tree with offspring distribution $\mu$. First the root is mapped to the origin under $\mathcal{S}$. Then, we assign to every edge $e$ of $T$ a random variable $Y_e$ according to $\theta$ independently and $\mathcal{S}(v)$, for any $v\in V(T)$ is the sum of the random variables $Y_e$ over all edges $e$ belonging to the unique simple path from the root to $u$ in the tree. Since we have an order $\prec$ in the offspring of each vertex in $T$, we have the classical order (according to Depth-first search) on $V(T)$ as follows. Let $v$ and $v'$ are different vertices, and $\omega=(v_0,v_1,\dots, v_m)$ and $\omega'=(v'_0,v'_1,\dots, v'_n)$ be the unique simple paths in the tree from the root (hence $v_0=v'_0$ is the root) to $v$ and $v'$ respectively. We say that $v$ is on the left of $v'$, i.e. $v\prec v'$ if either $(v_0,v_1,\dots, v_m)$ is a subsequence of $(v'_0,v'_1,\dots, v'_n)$ or $v_t\prec v'_t$, where $t=\min\{k: v_k\neq v'_k\}$.

For any branching random walk sample $S:V(T)\rightarrow \Z^d$ that visits $a$, $V_a:=\{v\in V(T):S(v)=a\}$ is not empty. Let $u$ be the leftmost point in $V_a$ and $(v_0,v_1,\dots,v_k)$ be the unique simple path in $T$ from the root to $u$. Then $(S(v_0),S(v_1),\dots,S(v_k))$ is a path in $\Z^4$ from the origin to $a$. We denote this path by $\tilde{\gamma}(S)$. Let $N$ be the number of visiting times of $a$. For any $\gamma$ be a path from the origin to $a$, define $p(\gamma)=P(N>0, \tilde{\gamma}(\mathcal{S})=\gamma )$ and $e(\gamma)=E(N|N>0,\tilde{\gamma}(\mathcal{S})=\gamma)$. Note that $N>0$ iff $\mathcal{S}$ visits $a$. For any path $\gamma=(z_0,\dots, z_n)$ in $\Z^4$, define $g(\gamma)=\sum_{i=0}^nG(z_i,a)=\sum_{i=0}^nG(a-z_i)$, where $G$ is the Green function respect to distribution $\theta$. Let $\mathcal{G}=1\{\mathcal{S}\; \text{visiting }a\}\cdot g(\tilde{\gamma}(\mathcal{S}))$. The following lemmas are the key ingredients for the main theorem.

\begin{lemma}\label{key_1}
For any $\gamma$ is a path from the origin to $a$ such that $p(\gamma)>0$, we have:
\begin{equation}
e(\gamma)\geq P(\mu\geq 2)g(\gamma).
\end{equation}
\end{lemma}

\begin{lemma}\label{key_2}
There exists positive constants $c_1,c_2$, such that for all $a\in\Z^4$ with $|a|$ sufficiently large, we have
\begin{equation}
P(0<\mathcal{G}\leq c_1\log |a|)\leq c_2/|a|^{2.1}.
\end{equation}
\end{lemma}

We postpone the proof of these two lemmas and start the proof for Theorem \ref{main}.
Since $\mu$ is critical, we have:
\begin{equation*}
EN=G(0,a)\asymp |a|^{-2}.
\end{equation*}
By Lemma \ref{key_1}, we have:
\begin{align*}
|a|^{-2}&\asymp EN\geq P(\mathcal{G}\geq c_1\log |a|)E(N|\mathcal{G}\geq c_1\log |a|)\\
&\geq P(\mathcal{G}\geq c_1\log |a|)P(\mu\geq2) c_1 \log |a|\\
&\succeq P(\mathcal{G}\geq c_1\log |a|)\log |a|.
\end{align*}
Therefore:
\begin{equation*}
P(\mathcal{G}\geq c_1\log |a|)\preceq 1/(|a|^2\log |a|).
\end{equation*}
Then we have:
\begin{align*}
P(\mathcal{S}\;\text{visiting }a)&=P(\mathcal{G}>0)\\
&=P(0<\mathcal{G}<c_1\log |a|)+P(\mathcal{G}\geq c_1 \log|a|)\\
&\preceq 1/|a|^{2.1}+1/(|a|^2\log |a|)\\
&\preceq 1/(|a|^2\log |a|).
\end{align*}

\begin{proof}[Proof of Lemma \ref{key_1}]
Fix a $\gamma=(z_0,z_1,\dots,z_k)$ such that $p(\gamma)>0$. For any branching random walk sample $S$ such that $\tilde{\gamma}(S)=\gamma$, write $a_i$ ($b_i$ respectively) for the number of the brothers of $z_i$ on the left of $z_i$ (on the right respectively), for $i=1,\dots,k$. From the tree structure, one can easily see that, for any $l_1,\dots,l_k,$ $m_1,\dots,m_k\in\N$, we have
\begin{multline}\label{f1}
P(N>0, \tilde{\gamma}(\mathcal{S})=\gamma; a_i=l_i, b_i=m_i, \text{for }i=1,\dots,k)\\
=s(\gamma)\prod_{i=1}^{k}\left(P(\mu=l_i+m_i+1)(q(z_{i-1}-a))^{l_i}\right),
\end{multline}
where $s(\gamma)$ is the probability weight for random walk respect to $\theta$, i.e. $s(\gamma)=\prod_{i=1}^k\theta(z_i-z_{i-1})$ and
$q(z)$ is the probability that the branching random walk does not visit $z$ conditioned on the initial particle having only one child.

Conditioned on the event on \eqref{f1}, the expectation of $N$ is:
\begin{equation*}
G(0)+\sum_{i=1}^k m_iG(a-z_{i-1}).
\end{equation*}
Recall that $g(\gamma)=\sum_{i=0}^{k}G(a-z_i)=G(0)+\sum_{i=1}^kG(a-z_{i-1})$. Thus it suffices to show:
\begin{equation*}\label{f2}
E(b_i|N>0, \tilde{\gamma}(\mathcal{S})=\gamma)\geq P(\mu\geq2).
\end{equation*}
A straight computation using \eqref{f1} gives:
\begin{align*}
E(b_i|N>0, \tilde{\gamma}(\mathcal{S})=\gamma)&=\frac{\sum_{l\geq0,m\geq0}mP(\mu=l+m+1)(q(z_{i-1}-a))^l}
{\sum_{l\geq0,m\geq0}P(\mu=l+m+1)(q(z_{i-1}-a))^l}\\
&\geq\frac{\sum_{l=0,m\geq1}1\cdot P(\mu=l+m+1)(q(z_{i-1}-a))^l}
{\sum_{l\geq0,m\geq0}P(\mu=l+m+1)}\\
&=\frac{\sum_{m\geq1}P(\mu=m+1)}
{\sum_{j\geq1}jP(\mu=j)}\\
&=\frac{P(\mu\geq2)}{E\mu}\\
&=P(\mu\geq2).\\
\end{align*}
\end{proof}

\begin{proof}[Proof of Lemma \ref{key_2}]
Straight calculation using \eqref{f1} gives:
\begin{align*}
p(\gamma)&=s(\gamma)\prod_{i=1}^{k}(\sum_{l_i\geq0}P(\mu\geq l_i+1)(q(z_{i-1}-a))^{l_i})\\
&\leq s(\gamma)\prod_{i=1}^{k}(\sum_{l_i\geq0}P(\mu\geq l_i+1))\\
&=s(\gamma)\prod_{i=1}^{k}(E\mu)\\
&=s(\gamma).
\end{align*}
Hence, we have:
\begin{align*}
P(0<\mathcal{G}\leq c_1 \log |a|)&=\sum_{\gamma:0\rightarrow a, g(\gamma)\leq c_1 \log |a|}p(\gamma)\\
&\leq\sum_{\gamma:0\rightarrow a, g(\gamma)\leq c_1 \log |a|}s(\gamma)\\
&=P^{RW}(0<\mathcal{G}\leq c_1 \log |a|),\\
\end{align*}
where $P^{RW}$ is the probability about Random Walk with step distribution $\theta$.
Then Lemma \ref{key_2} can be implied by:

\begin{prop}
There exist $c_1,c_2$ such that for $a\in\Z^4$ with $|a|$ sufficiently large,
\begin{equation*}
P^{RW}(\tau_a<\infty, \sum_{i=0}^{\tau_a}G(S_i)\leq c_1 \log |a|)\leq c_2 |a|^{-2.1},
\end{equation*}
where $(S_i)_{i\in\N}$ is Random Walk starting from $0$ with distribution $\theta$ and $\tau_a$ is the hitting time for $a$.
\end{prop}

Note that we have changed $G(\cdot, a)$ to $G(\cdot)$. We can do this by considering the reversed random walk.

This proposition is an adjusted version of Lemma 10.1.2 (a) in \cite{LL10}. It is assumed there that $\theta$ has finite support which is much stronger than our case, though its conclusion is also much stronger than ours. The argument is similar to the one there with small adjustments. We give an outline here. It suffices to prove:
\begin{equation}\label{f3}
P^{RW}(\sum_{i=0}^{\tau_n}G(S_i)\leq c_1 \log n)\leq c_2 n^{-2.1},
\end{equation}
where $\tau_n=\min\{k\geq0: |S_k|\geq n\}$.

Fix $\beta=0.9\in(4/5,1)$ and let $N=\lfloor n^\beta\rfloor$. Since $\theta$ has finite fifth moment, we have $P(|\theta|>m)\preceq m^{-5}$. Let $A$ be the event that $|X_i|\leq  N$, for $i=1,2,\dots,2n^2\wedge\tau_n$ (where $X_i=S_i-S_{i-1}$ and $i\wedge j=\min\{i,j\}$). Then $P(A^c)\preceq n^2/n^{5\beta}\leq n^{-2.1}$. When $A$ happens, the range of the random walk is bounded by $N$ for the first $2n^2$ steps. Since only first $2n^2$ steps is bounded, we need to change the stopping times there a little. Let $\xi^0=0$, $\xi^i=\min\{k:|S_k|\geq 2^{i}N\} \wedge (\xi^{i-1}+(2^iN)^2)$, for $i=1,2,\dots, L$, where $L=\max\{k:2^kN\leq n\}\asymp \log n$. Now \eqref{f3} can be obtained by following the argument of the proof of Lemma 10.1.2 (a) in \cite{LL10} .
\end{proof}

\section*{Acknowledgement}

The author would like to thank his advisor, Professor Omer Angel for inspiring discussions.

\end{document}